\documentclass[11pt]{amsart}
\usepackage[english]{babel}
\usepackage{amssymb,enumerate,amsmath,txfonts,tikz}
 \usepackage{graphicx}
\usepackage[colorlinks=true,linkcolor=blue,citecolor=blue,urlcolor=blue]{hyperref}
\numberwithin{equation}{section}
\usepackage{tikz}
\usetikzlibrary{arrows,shapes,trees,backgrounds}

\newtheorem{theo}{Theorem}[section]
\newtheorem{pro}[theo]{Proposition}
\newtheorem{lem}[theo]{Lemma}

\newtheorem{conjecture}[theo]{Conjecture}
\newtheorem{prob}[theo]{Problem}

\theoremstyle{remark}

\newtheorem{rem}[theo]{Remark}

\renewcommand{\(}{\left(}
\renewcommand{\)}{\right)}

\newcommand{\R}{\mathbb{R}}

\renewcommand{\d}{\delta}

\renewcommand{\l}{\lambda}

\newcommand{\ra}{\rightarrow}

\newcommand{\Vol}{\operatorname{Vol}}
\newcommand{\dvol}{\operatorname{dvol}}

\newcommand{\tr}{\operatorname{tr}}

\newcommand{\divv}{\operatorname{div }}

\newcommand{\shd}{\operatorname{sh}_\delta}
\newcommand{\chd}{\operatorname{ch}_\delta}
\newcommand{\thd}{\operatorname{th}_\delta}

\begin{document}
\title[Total squared mean curvature]{Total squared mean curvature of immersed submanifolds in a negatively curved space}
\author{Yanyan Niu\textsuperscript{*}}
\thanks{\textsuperscript{*}Supported in part by National Natural Science Foundation of China [11821101]}
\address{Yanyan Niu \newline \indent School of Mathematical Sciences \newline \indent Capital Normal University, Beijing 100048, China}
\email{\href{mailto:yyniukxe@gmail.com}{yyniukxe@gmail.com}}

\author{Shicheng Xu\textsuperscript{\dag}}
\thanks{\textsuperscript{\dag}Supported in part by Beijing Natural Science Foundation [Z19003] and the Youth Innovative Research Team of Capital Normal University}

\address{Shicheng Xu\newline \indent School of Mathematical Sciences\newline \indent Capital Normal University, Beijing 100048, China}
\address{Academy for Multidisciplinary Studies
\newline \indent Capital Normal University, Beijing 100048,  China}
\email{\href{mailto:shichengxu@gmail.com}{shichengxu@gmail.com}}

\date{\today}
\begin{abstract}
Let $n\ge 2$ and $k\ge 1$ be two integers.
Let $M$ be an isometrically immersed closed $n$-submanifold of co-dimension $k$ that is homotopic to a point in a complete manifold $N$, where the sectional curvature of $N$ is no more than $\d<0$. We prove that the total squared mean curvature of $M$ in $N$ and the first non-zero eigenvalue $\lambda_1(M)$ of $M$ satisfies
$$\lambda_1(M)\le n\(\d +\frac{1}{\Vol M}\int_M |H|^2 \dvol\).$$
The equality implies that $M$ is minimally immersed in a metric sphere after lifted to the universal cover of $N$. This completely settles an open problem raised by E. Heintze in 1988.
\end{abstract}
{\maketitle}

\section{Introduction}
Problems on the Willmore energy, i.e. the integral of the square of the mean curvature of submanifolds, in the Euclidean space and other space forms have a long history, around which many important theories and tools have been developed. This energy raised up naturally in the study of vibrating properties of thin plates in the 1810s (see the survey \cite{Marques-Neves2014-2}). The Willmore conjecture, proposed by T. J. Willmore \cite{Willmore1965} in 1965, states that the Willmore energy has minimum $2\pi^2$ among all tori immersed in the Euclidean three space $\mathbb R^3$. 
In 1982, Li and Yau \cite{Li-Yau1982} proved that the Willmore energy of non-embedded surface is at least $8\pi>2\pi^2$.
Marques and Neves \cite{Marques-Neves2014} finally confirmed this conjecture in 2014. 

Meanwhile, the generalization of classical results in $\mathbb R^n$ to a Cartan-Hadamard manifold, such as the isoperimetric inequality, has attracted a lot of attention over years. A Cartan-Hadamard manifold is a complete simply connected Riemannian manifold with nonpositive sectional curvature. Whether the Euclidean isoperimetric inequality holds in a Cartan-Hadamard manifold $N$, is known as the Cartan-Hadamard conjecture. It can trace back to \cite{Weil1926} in 1926 for surfaces, and for higher dimensions \cite{Aubin1975}, \cite{Gromov1981} and \cite{Burago-Zalgaller1980, Burago-Zalgaller1988}. Developments including the most recent refer to \cite{Croke1984}, \cite{Kleiner1992}, \cite{Ghomi-Spruck2021} and references therein. Optimal lower bounds for the Willmore energy of submanifolds also provide key ingredients in the study of isoperimetric inequality in a Cartan-Hadamard manifold of negative curvature \cite{Schulze2020}.  

This paper is to positively answer a problem on the first nonzero eigenvalue and the Willmore energy of immersed closed manifolds in a negatively curved Cartan-Hadamard manifold, which was raised by Heintze \cite{Heintze1988} in 1988.

Let $M$ be a closed and connected Riemannian $n$-manifold isometrically immersed into $\mathbb{R}^{n+k}$. Let $\lambda_1(M)$ be the first eigenvalue of the Laplace operator of $M$.
It was firstly shown by Reilly \cite{Reilly1977} (cf. \cite{BW1976}) that
\begin{equation}\label{ineq-Reilly}
\lambda_1(M)\le \frac{n}{\Vol M}\int_M |H|^2\dvol,
\end{equation}
where $H=\frac{1}{n}\tr \alpha$ is the mean curvature vector field along $M$, and $\alpha$ the second fundamental form of $M$ in $N$. 

In \cite{Heintze1988} Heintze considered the general case that $N$ is a complete Riemannian manifold and $M$ is immersed into a normal ball $B(p,R)$ in $N$. Let $\d$ be the supremum of sectional curvature on $B(p,R)$. When $\d$ is nonnegative and $R \sqrt{\d}\le \frac{\pi}{4}$, he proved that
\begin{equation}\label{ineq-Heintze}
\lambda_1(M)\le n\left(\d+\frac{1}{\Vol M}\int_M |H|^2\dvol\right).
\end{equation}

Heintze observed that \eqref{ineq-Heintze} generally fails for curves in the hyperbolic plane, due to the work of Langer and Singer \cite{Langer-Singer1984}. A weaker inequality was proved in \cite{Heintze1988} for negative $\d$,
\begin{equation}\label{ineq-Weak}
\lambda_1(M)\le n\left(\d+\max_{M} |H|^2\right).
\end{equation}
Based on the fact showed in \cite{Langer-Singer1984} that \eqref{ineq-Heintze} holds for curves in $\mathbb H^2$ of length $\le 2\pi$, Heintze asked in the section \cite[\S 4]{Heintze1988} devoted to open problems that 

\textit{Does \eqref{ineq-Heintze} hold for ``small'' submanifolds $M$ in a negatively curved manifold $N$?}

In 1992 Soufi and Ilias \cite{Soufi-Ilias1992} affirmatively solved the case when $N$ is the hyperbolic space $\mathbb H^{n+k}$. In fact, they proved that \eqref{ineq-Heintze} holds for all $n$-submanifolds with $n\ge 2$ in any Riemannian manifold conformal to $\mathbb H^{n+k}$. 

The main result of this paper gives a positive answer of the above problem for the general case, i.e., all submanifolds of dimension $n\ge 2$ without smallness restriction in a negatively curved Cartan-Hadamard manifold.

\begin{theo}\label{thm-sharp-lambda1}
	The inequality \eqref{ineq-Heintze} holds for any isometrically immersed closed submanifold $M$ of dimension $n\ge 2$ and co-dimension $k\ge 1$ in a Cartan-Hadamard manifold $N$ with sectional curvature $\le \d<0$.
	
	Equality in \eqref{ineq-Heintze} implies that $M$ is minimally immersed into a metric sphere of radius $\operatorname{arcsh}_\d\sqrt{\frac{n}{\lambda_1(M)}}$, where $\operatorname{arcsh}_\d$ is the inverse function of the $\d$-hyperbolic sine $\shd(r)=\frac{1}{\sqrt{-\d}}\sinh\(\sqrt{-\d}r\)$.
\end{theo}

Theorem \ref{thm-sharp-lambda1} yields an optimal lower bound  $\(\frac{\lambda_1(M)}{n}-\d\)\Vol M$ for the Willmore energy. As what had already known in \cite{Soufi-Ilias1992}, if $N=\mathbb H^{n+k}$, then by \cite{Takahashi1966}, all submanifolds for which \eqref{ineq-Heintze} takes equality coincides exactly with those $M$ isometrically immersed in a Euclidean sphere $S^{n+k-1}$, whose canonical coordinate functions in $\mathbb R^{n+k}$ are the first eigenfunctions of the Laplacian of $M$. The immersion from $M$ to $\mathbb H^{n+k}$ is realized through a totally umbilical embedding of $S$ as a metric sphere in $\mathbb H^{n+k}$. Irreducible homogeneous spaces, a Clifford torus, an equilateral torus, etc. provide such examples (\cite{Takahashi1966}, cf. \cite{Lawson1980}). It is equivalent to a conjecture by Yau \cite[Problem 100]{Yau1982} whether \eqref{ineq-Heintze} is an equality for all embedded minimal hypersurfaces in $S^{n+k}$.

Let $C_n(N^{n+k})=\sup_M\(\lambda_1(M)-\frac{n}{\Vol M}\int_M|H|^2\dvol\)$ be the constant introduced by Heintze in \cite{Heintze1988}, where the
supremum is taken over all $n$-dimensional closed isometrically
immersed submanifolds $M$ of $N$. By \eqref{ineq-Heintze} for $\d=0$, Heintze proved that $C_n(N^{n+k})\le 0$ if $N$ is a Cartan-Hadamard manifold. It follows from Theorem \ref{thm-sharp-lambda1} that if the sectional curvature of $N$ $\le \d<0$, then $C_{n}(N^{n+k})\le n\d<0$ for any $n\ge 2$ and $k\ge 1$. It was already known that $C_n(\mathbb H^{n+k})=-n$ by \cite{Soufi-Ilias1992} for $n\ge 2$ and $C_1(\mathbb H^{n+k})=0$ by \cite{Langer-Singer1984}.

Our method also implies that the ``smallness'', which is originally conjectured by Heintze, does make sense for all regular closed curves.
For a closed curve $\gamma$ of length $l$ in a manifold with sectional curvature $\le \d<0$, $\lambda_1(\gamma)=\frac{4\pi^2}{l^2}$. Hence,
\eqref{ineq-Heintze} is equivalent to
\begin{equation}\label{ineq-total-squared-curvature}
\int_\gamma \kappa^2 ds\ge \frac{4\pi^2}{l} - \d l.
\end{equation}
As mentioned before, \eqref{ineq-total-squared-curvature} generally fails for curves with a large length; see a curve in $\mathbb H^2$ shown by \cite[Figure 8]{Langer-Singer1984} whose length can be a very large value, while the total squared curvature keeps bounded. 

\begin{theo}\label{thm-total-squared-curvature}
	Any regular closed curve $\gamma$ of length $l\le \frac{2\pi}{\sqrt{-\d}}$ in a Cartan-Hadamard $N$ with sectional curvature $\le \d<0$ satisfies \eqref{ineq-total-squared-curvature}.
	
	Among those curves, equality in \eqref{ineq-total-squared-curvature} implies that $\gamma$ lies in a metric sphere $S\subset N$ of radius $\operatorname{arcsh}_\d\frac{L}{2\pi}$, and is a geodesic of $S$.
\end{theo}

Let $\gamma$ be a null-homotopic regular closed curve in a $2$-manifold with constant curvature $\d<0$. As one of the main results in \cite{Langer-Singer1984}, an optimal lower bound of the total squared geodesic curvature along $\gamma$ was proved by Langer and Singer, 
\begin{equation}\label{ineq-Langer-Singer}
\int_\gamma \kappa^2 ds\ge 4\pi \sqrt{-\d}.
\end{equation}
They conjectured in the same paper that \eqref{ineq-Langer-Singer} generally holds for null-homotopic curves in a negatively curved manifold, i.e.,

\textit{Let $N$ be a Riemannian manifold with sectional curvature
bounded above by $\d<0$. If there exists a regular closed curve $\gamma$ in $N$ satisfying $\int_\gamma \kappa^2 ds< 4\pi \sqrt{-\d}$, then $N$ cannot be simply connected.}

Ulrich Pinkall proved \eqref{ineq-Langer-Singer} in hyperbolic spaces of any dimension \cite{Langer-Singer1984}. Theorem \ref{thm-total-squared-curvature} directly implies \eqref{ineq-Langer-Singer} for curves with length $\le \frac{2\pi}{\sqrt{-\d}}$ in a Cartan-Hadamard manifold. When $l>\frac{2\pi}{\sqrt{-\d}}$, by Proposition \ref{pro-bound-large-length} below, $\int_{\gamma}\kappa^2 ds> \frac{8\pi^2}{l}$. 

We conjecture that Langer and Singer's original statement holds, which now can be reformed as follows. 

\begin{conjecture}\label{conj-main}
	Any null-homotopic regular closed curve $\gamma$ of length $>\frac{2\pi}{\sqrt{-\d}}$ in a manifold with sectional curvature $\le \d<0$ satisfies $\int_\gamma \kappa^2 ds\ge 4\pi \sqrt{-\d}$.
\end{conjecture}

\section{Proof of Theorem \ref{thm-sharp-lambda1}}

Let $\d<0$ be a negative real number. Throughout the paper we use  $\shd(r)=\frac{1}{\sqrt{-\d}}\sinh(\sqrt{-\d}r)$ to denote the $\d$-hyperbolic sine, $\chd(r)=\shd'(r)$ the $\d$-hyperbolic cosine, and $\thd(r)=\frac{\shd(r)}{\chd(r)}$ the $\d$-hyperbolic tangent function. They satisfy the identities
$\chd'=-\d \shd$, $\chd^2+\d\shd^2=1$ and $1+\d\thd^2=\frac{1}{\chd^2}$. 

The key ingredient in the proof of Theorem \ref{thm-sharp-lambda1} is an optimal lower bound on the Willmore energy in terms of the $L^2$ integral of $\d$-hyperbolic tangent $\thd(r)$.

\begin{pro}[$L^2$ lower bound of mean curvature]\label{prop-sharp-l2-mean-curv}
	Let $M$ be an immersed closed $n$-submanifold of co-dimension $k\ge 1$ in a normal ball $B(p,R)$ of a Riemannian manifold $N$, where the sectional curvature of $B(p,R)$ satisfies $K_{B(p,R)}\le \d<0$. Let $H$ be the mean curvature vector of $M$ in $N$, and Let $r(x)=d(p,x)$ be distance function to $p$ in $N$. Then 
	\begin{equation}\label{ineq-sharp-l2-mean-curv}
	\int_M |H|^2 \dvol \cdot \int_M \thd^2(r) \dvol\ge \Vol^2(M),
	\end{equation}
	if one of the followings holds,
	\begin{enumerate}\numberwithin{enumi}{theo}
		\item\label{prop-sharp-l2-1} either $n\ge 2$, or
		\item\label{prop-sharp-l2-2} $n=1$ and $\shd(R)\le \frac{1}{\sqrt{-\d}}$ ( equivalently, $\sqrt{-\d}\cdot R\le \frac{1}{2}\ln \(3+2\sqrt{2}\)$).
	\end{enumerate} 

	Moreover, equality in \eqref{ineq-sharp-l2-mean-curv} implies that $M$ is minimally immersed in a metric sphere of radius $R_0$, where $\thd^2(R_0)=\frac{\Vol M}{\int_M |H|^2\dvol}$.
\end{pro}
\begin{rem}
	If $M$ is an immersed hypersurface, then it can be easily seen from the proof of Proposition \ref{prop-sharp-l2-mean-curv} and the rigidity of Hessian comparison for distance functions that the equality of \eqref{ineq-sharp-l2-mean-curv} or \eqref{ineq-Heintze} implies that $M$ is an embedded sphere of constant curvature $\frac{1}{\shd^2(R_0)}=\d + \frac{1}{\Vol M}\int |H|^2$, and the domain enclosed by $M$ in $N$ is a ball of constant curvature $-\d$; cf. \cite{Hu-Xu2019,Hu-Xu2020}.
\end{rem}

By assuming Proposition \ref{prop-sharp-l2-mean-curv}, let us first prove Theorem \ref{thm-sharp-lambda1}. Let $M$ be a closed $n$-submanifold immersed in a convex domain $U$ in a complete Riemannian $(n+k)$-manifold $N$, where the upper curvature bound of $U$, $\d=\sup_{U}K_U\le 0$.
Let $\mathcal{F}:U\ra \R$ be an energy function defined by
\begin{align*}
\mathcal{F}(q):=\int_{M} \Phi_\d(d(q,x)) \dvol(x),
\end{align*}
where $d(q,x)$ is the distance between $q$ and $x$ in $U$, and $\Phi_\d$ is the modifier defined by
\begin{equation*}
\Phi_\d(r):=\int_0^r \thd(s)ds.
\end{equation*}

By the standard Hessian comparison for distance functions, every $\Phi_\d(d(x,\cdot))$ is convex in $U$. Hence $\mathcal{F}$ is a strictly convex function on $U$, which admits a unique interior minimum point $p_0$ in $U$.
We call $p_0\in N$ the {\em center of mass} of $M$ with respect to the modified distance $\Phi_\d(r)$.

Note that the gradient of $\mathcal{F}$ has expression $\nabla^N \mathcal F(q)=-Y(q)$, where
\begin{align}\label{3.6}
Y(q):=\int_M \frac{\thd(r_q)}{r_q}\exp_q^{-1}(x)\in T_qU,
\end{align}
and $r_q(x)=d(x,q)$.
Since $\nabla^N \mathcal F(p_0)=0$, in the normal coordinates $\{x_1,\cdots,x_{n+k}\}$ of $p_0$, \eqref{3.6} becomes
\begin{align}\label{center-mass-test-func}
\int_M \frac{\thd(r)}{r}x_i=0, \qquad i=1,\dots,n+k,
\end{align}
where $r(x)=d(x,p_0)$. For simplicity we have dropped $\dvol$ in the expression of integrals over $M$.

Let the {\em position vector field} $X$ along $M$ about $p_0$ defined to be the gradient of $\Phi_\d(r)$, i.e.,
\begin{equation}\label{ineq-def-position-vec}
X=\nabla^N\Phi_\d(r)=\thd(r)\nabla^Nr=\frac{\shd(r)}{\chd(r)}\nabla^Nr.
\end{equation} 
Then by \eqref{center-mass-test-func}, each component of $X$, $\frac{\thd(r)}{r}x_i$, provides a test function of $\lambda_1(M)$. We will use the renormalized $L^2$ norm of $X$ defined by 
\begin{equation}\label{def-renorm-l2}
\|X\|_2=\(\frac{1}{\Vol(M)}\int_M |X|^2 \)^{\frac{1}{2}}.
\end{equation}

The following lemma will be applied in the proofs below, where
\eqref{ineq-divM-top} is for Proposition \ref{prop-sharp-l2-mean-curv}, and \eqref{ineq-nabla-coordinate-f} for Theorem \ref{thm-sharp-lambda1}.
\begin{lem}[cf. \cite{Heintze1988}]\label{lem-tech-f} Let the assumptions be as in Proposition \ref{prop-sharp-l2-mean-curv}. Let $f(r)$ be a smooth function. Then
	\begin{equation}\label{ineq-divM-top}
	\divv_M \(f\nabla^N r\)^\top \geq n f \cdot \frac{\chd(r)}{\shd(r)}+nf\langle \nabla^N r,H\rangle +\(f'-f \cdot \frac{\chd(r)}{\shd(r)}\) |\nabla^M r|^2,
	\end{equation}
	where $\(f\nabla^N r\)^\top$ is the component tangent to $M$ and the {\em divergence} of a vector field $Y$ along $M$ is defined by
	$$
	\divv_M Y(x):=\sum_{i=1}^{n}\langle \nabla^N_{e_i}Y,e_i\rangle, \; \forall x\in M  \text{ and an orthonormal basis $e_1,\cdots,e_n$ of $T_x M$}.
	$$
	
	\begin{equation}\label{ineq-nabla-coordinate-f}
	\sum_{i=1}^{n+k}\left|\nabla^M \(\frac{f(r)}{r}x_i\)\right|^2 \leq n\frac{f^2}{\shd^2}+
	\(\(f'\)^2-\frac{f^2}{\shd^2}\) \left|\nabla^M r\right|^2.
	\end{equation}
\end{lem}
Since Lemma \ref{lem-tech-f} follows directly from those calculations in the proof of \cite[Lemmas 2.4, 2.7]{Heintze1988}, we omit its proof.

\begin{proof}[Proof of Theorem \ref{thm-sharp-lambda1}]
	~
	
	Let $p_0$ be the center of mass of $M$ with respect to the modified distance $\Phi_\d(r)$. Then $p_0\in B(p,R)$. Let $X$ be the position vector field along $M$ about $p_0$ given in \eqref{ineq-def-position-vec}.
	
	Let us apply Rayleigh's principle to each component of $X$ and then sum them together, 
	\begin{align}\label{1st-eigen-test}
	\l_1 \int_M |X|^2=\l_1 \int_M \sum_{i=1}^{n+k}\left(\frac{\thd(r)}{r} x_i\right)^2 \leq & \int_M \sum_{i=1}^{n+k}\left|\nabla^M \frac{\thd(r)}{r} x_i\right|^2.
	\end{align}
	
	At the same time, let us take $f=\thd$ in \eqref{ineq-nabla-coordinate-f}. Then
	\begin{equation}\label{ineq-nabla-test-func-est}
	\sum_{i=1}^{n+k}\left|\nabla^M \(\frac{\thd(r)}{r}x_i\) \right|^2\leq \frac{n}{\chd^2(r)} = n+ n\d\thd^2(r).
	\end{equation}
	
	Combining \eqref{1st-eigen-test} and \eqref{ineq-nabla-test-func-est} with \eqref{ineq-sharp-l2-mean-curv}, we derive \eqref{ineq-Heintze} by
	\begin{equation}\label{ineq-1st-eigen-mean-curv}
	\lambda_1\|X\|_2^2 \le n \(\d \|X\|_2^2+1\)\le n\(\d+\|H\|_2^2\)\|X\|_2^2.
	\end{equation}
		
	If equality holds in \eqref{ineq-Heintze}, then so is \eqref{ineq-sharp-l2-mean-curv}. It follows from Proposition \ref{prop-sharp-l2-mean-curv} that $M$ is minimally immersed in the metric sphere of radius $\operatorname{arcsh}_\d\sqrt{\frac{n}{\lambda_1(M)}}$.
	
	Now the proof of Theorem \ref{thm-sharp-lambda1} is complete.
\end{proof}

\section{Lower bound of Willmore energy}
In this section let us prove Proposition \ref{prop-sharp-l2-mean-curv}.

Let $X$ be the position vector defined by \eqref{ineq-def-position-vec}. Since $\(\nabla^N r\)^\top=\nabla^Mr$, the component of $X$ tangent to $M$ is $X^\top=\thd(r)\nabla^Mr$.
Let us take $f=\thd$ in \eqref{ineq-divM-top}, then it is transformed to
\begin{equation}\label{ineq-div-position}
\operatorname{div}_M (X^\top)\ge n+n\left<X,H\right>+\delta \left|X^\top\right|^2.
\end{equation}
Integrating \eqref{ineq-div-position}, by the divergence theorem we have
\begin{equation}\label{ineq-int-div-position}
\int_M \left<-X^\bot,H\right> \ge \Vol M +\frac{\d}{n}\int_M|X^\top|^2.
\end{equation}

By the Cauchy-Schwartz inequality, \eqref{ineq-int-div-position} yields the following key estimate in the proof of Proposition \ref{prop-sharp-l2-mean-curv}.
\begin{lem}[Key lemma]\label{lem-l2-mean-curv-key}
	\begin{align}\label{ineq-l2-mean-curv-key}
	\|X^\bot\|_2\cdot \|H\|_2 &\ge 1+\frac{\d}{n}\|X^\top\|_2^2,
	\end{align}
	where the renormalized $L^2$ norm of vector fields defined by \eqref{def-renorm-l2} are used.
\end{lem}

	Since $\d\le 0$ and $|X|\le \thd(R)$, 
	$$\thd(R)\cdot \|H\|_2 \ge 1+\frac{\d}{n} \thd^2(R),$$
	a rough lower bound of $\|H\|_2$ follows from Lemma \ref{lem-l2-mean-curv-key}
	\begin{equation}\label{ineq-l2-mean-curv-rough}
		\|H\|_2 \ge \frac{1}{\thd(R)} +  \frac{\d}{n}\thd(R).
	\end{equation}
	Note that, for $\d\le 0$ and $n\ge 1$, $$\frac{1}{\thd(R)} +  \frac{\d}{n}\thd(R)\ge \frac{1}{\thd(R)} +  \d\thd(R)=\frac{1}{\shd(R)\chd(R)}>0,$$ which implies that the rough lower bound always makes sense.	

Proposition \ref{prop-sharp-l2-mean-curv} can be viewed as a sharp lower bound estimate \eqref{ineq-sharp-l2-mean-curv} improved from \eqref{ineq-l2-mean-curv-rough}. In the following we write \eqref{ineq-sharp-l2-mean-curv} in the form $\|X\|_2^2 \cdot \|H\|_2^2\ge 1$.

\begin{proof}[Proof of Proposition \ref{prop-sharp-l2-mean-curv}]
~

Let us first observe the right hand side of \eqref{ineq-l2-mean-curv-key} is always positive,
\begin{align*}
1+\frac{\d}{n}\|X^\top\|_2^2
\ge 1+\frac{\d}{n} \thd^2(R)
 \ge \frac{1+\d\thd^2(R)}{n}=\frac{1}{n\chd^2(R)}>0,
\end{align*}
which implies that inequality is preserved after taking square of both sides of \eqref{ineq-l2-mean-curv-key}, i.e.,
$$\(\|X\|_2^2-\|X^\top\|_2^2\) \cdot \|H\|_2^2\ge  \(1+\frac{\d}{n}\|X^\top\|_2^2\)^2\ge 1+2\frac{\d}{n}\|X^\top\|_2^2,$$
Hence we have
\begin{align*}
\|X\|_2^2 \cdot \|H\|_2^2 &\ge 1+\(2\frac{\d}{n}+\|H\|_2^2\)\|X^\top\|_2^2\\
& \ge 1+\(2\frac{\d}{n}+\(\frac{1}{\thd(R)} +  \frac{\d}{n}\thd(R)\)^2\)\|X^\top\|_2^2. \qquad \text{by \eqref{ineq-l2-mean-curv-rough}}
\end{align*}
Note that for $n\ge 4$,
\begin{align*}
2\frac{\d}{n}+\(\frac{1}{\thd(R)} +  \frac{\d}{n}\thd(R)\)^2 & =4\frac{\d}{n}+\frac{1}{\thd^2(R)} +  \frac{\d^2}{n^2}\thd^2(R)\\
& \ge \d +\frac{1}{\thd^2(R)}=\frac{1}{\shd^2(R)}>0.
\end{align*}
Thus we derive $\|X\|_2^2 \cdot \|H\|_2^2\ge 1$ for $n\ge 4$.

More careful investigation is required for $n=2$ and $3$. Let us reform \eqref{ineq-l2-mean-curv-key} to 
$$\|X^\bot\|_2\cdot \|H\|_2\ge 1+\frac{\d}{n}\(\|X\|_2^2-\|X^\bot\|_2^2\).$$ Then we have
$$\(\|H\|_2+\frac{\d}{n}\|X^\bot\|_2\)\|X^\bot\|_2\ge 1+\frac{\d}{n}\|X\|_2^2.$$

Note that when $\d< 0$, $\varphi_h(x)=(h+\frac{\d}{n}x)x$ is increasing for $x\le \frac{nh}{-2\d}$.
 We claim that for $n\ge 2$,
\begin{equation}\label{ineq-l2-mean-curv-claim}
\|X\|_2+\|X^\bot\|_2\le \frac{n}{-\d}\|H\|_2 \qquad \text{(and thus $\|X^\bot\|_2\le \frac{n}{-2\d}\|H\|_2$)}.
\end{equation}  
Since $\varphi_{h}(x)$ for $h=\|H\|_2$ takes minimum at the boundary points of the interval $\left[\|X^\bot\|_2, \frac{nh}{-\d}-\|X^\bot\|_2\right],$ and by \eqref{ineq-l2-mean-curv-claim}, $\|X\|_2$ lies in the interval,
$$\|H\|_2\cdot \|X\|_2+\frac{\d}{n}\|X\|_2^2\ge \|H\|_2\cdot \|X^\bot\|_2+\frac{\d}{n}\|X^\bot\|_2^2\ge  1+\frac{\d}{n}\|X\|_2^2,$$
which implies $\|H\|_2\cdot \|X\|_2\ge 1$.

We now verify the claim \eqref{ineq-l2-mean-curv-claim}. Indeed, for $n\ge 2$ we have
$$-\d \|X\|_2^2-1\le -\d \thd^2(R)-1=-\frac{1}{\chd^2(R)}<0\le \frac{n}{2}-1.$$
It follows that
$$\|X\|_2\cdot \|X^\bot\|_2\le \|X\|_2^2 \le \frac{n}{-2\d}.$$
Then we derive \eqref{ineq-l2-mean-curv-claim} by the inequalities below,
\begin{align}
\label{eq-prop-verification}
\frac{-\d}{n}\(\|X\|_2+\|X^\bot\|_2\)\cdot \|X^\bot\|_2 
& =  \frac{-\d}{n} \|X\|_2 \cdot \|X^\bot\|_2 + \frac{-\d}{n}\|X\|_2^2 +  \frac{\d}{n}\|X^\top\|_2^2 \\ \nonumber
& \le  1+\frac{\d}{n}\|X^\top\|_2^2\\
\nonumber
& \le \|H\|_2\cdot \|X^\bot \|_2 . \qquad \text{by \eqref{ineq-l2-mean-curv-key}}
\end{align}
Now the proof of Proposition \ref{prop-sharp-l2-mean-curv} for $n\ge 2$ is complete.

What remains is to prove for $n=1$ and $\shd(R)\le \frac{1}{\sqrt{-\d}}$. By setting $n=1$, it can be easily seen from the above verification of the claim that, together with \eqref{eq-prop-verification} and \eqref{ineq-l2-mean-curv-key}, the following condition
\begin{equation}\label{ineq-sharp-sufficient}
\|X\|_2 \cdot \|X^\bot\|_2 + \|X\|_2^2\le \frac{1}{-\d},
\end{equation}
implies \eqref{ineq-l2-mean-curv-claim}, and hence $\|X\|_2^2 \cdot \|H\|_2^2\ge 1$.

On the other hand, \eqref{ineq-sharp-sufficient} holds when $\chd^2(R)\le 2$, or equivalently $\shd(R)\le \frac{1}{\sqrt{-\d}}$. This is because  
$$2\|X\|_2^2\le 2\thd^2(R)=\frac{2-\frac{2}{\chd^2(R)}}{-\d}\le  \frac{1}{-\d}.$$ 

Thus
the inequality $\|X\|_2^2 \cdot \|H\|_2^2\ge 1$ holds for $n=1$ and $R$ satisfying $\shd(R)\le \frac{1}{\sqrt{-\d}}$, which is equivalent to $$\sqrt{-\d}\cdot R\le \frac{1}{2}\ln \frac{\sqrt{2}+1}{\sqrt{2}-1}=\frac{1}{2}\ln \(3+2\sqrt{2}\).$$

To complete the proof of Proposition \ref{prop-sharp-l2-mean-curv}, what remains is to check what happens when the inequalities above hold as equality. 

First, the auxiliary function $\varphi_h$ takes the same value on $\|X\|_2$ and $\|X^\bot\|_2$. Either $\|X\|_2=\|X^\bot\|_2$ or $\|X\|_2=\frac{nh}{-\d}-\|X^\bot\|_2$ implies that  $X\equiv X^\bot$ and $X^\top\equiv 0$ along $M$. Thus $X$ admits a constant norm and $M$ lies in a metric sphere $S$ centered at $p$.

At the same times, by the equality of Cauchy-Schwartz inequality, when \eqref{ineq-int-div-position} and \eqref{ineq-l2-mean-curv-key} take equalities, it follows that $H$ is parallel to $X$, which means that $M$ is also minimal in $S$.

Finally, since $|X|=\thd(r)=\thd(R_0)$ is a constant along $M$, the radius $R_0$ of $S$ satisfies $\thd^2(R_0)=\|X\|_2^2=\frac{1}{\|H\|_2^2}$.

\end{proof}

\section{Total squared curvature of curves}

In this section we analyze when \eqref{ineq-total-squared-curvature} holds for a regular closed curve in a manifold of negative curvature.

First, let us observe that, if a regular closed curve $\gamma$ lies in a normal ball $B(p,R)$ with $R\le  \operatorname{arcsh}_\d(1/\sqrt{-\d})$, then $\|X\|_2^2\le \thd(R)\le \frac{1}{-2\d}$. By the proof of \eqref{prop-sharp-l2-2} in Proposition \ref{prop-sharp-l2-mean-curv} $\|X\|_2^2\le \frac{1}{-2\d}$ implies \eqref{ineq-sharp-sufficient}. Hence \eqref{ineq-sharp-l2-mean-curv} always holds for such a curve regardless of its length. If $p$ is the center of mass of $\gamma$, then by the proof of Theorem \ref{thm-sharp-lambda1}, \eqref{ineq-total-squared-curvature}  holds.

The proof of Theorem \ref{thm-total-squared-curvature} is to show that $\|X\|_2^2\le \frac{1}{-2\d}$ also holds when the length of $\gamma$ is no more than $\frac{2\pi}{\sqrt{-\d}}$.

\begin{proof}[Proof of Theorem \ref{thm-total-squared-curvature}]
	~
	
	Let $\gamma$ be a regular closed curve with length $l$, and $X$ be the position vector about the center of mass for $\gamma$. We now show that the averaged position $\|X\|_2$ of $\gamma$ is also under control by
	\begin{equation}\label{ineq-length-position}
	\|X\|_2^2\le \frac{1}{\frac{4\pi^2}{l}-\d}.
	\end{equation}
	Indeed, by \eqref{ineq-1st-eigen-mean-curv} for $n=1$, we have
	$$\lambda_1\|X\|_2^2 \le 1+\d \|X\|_2^2,$$
	i.e., $\|X\|_2^2\le \frac{1}{\lambda_1-\d}$. Since the first eigenvalue of a curve is related to its length by $\lambda_1\cdot l^2=4\pi^2$, \eqref{ineq-length-position} is derived. 
	
	In particular, if $l\le \frac{2\pi}{\sqrt{-\d}}$, then $\|X\|_2^2\le \frac{1}{-2\d}.$ So by the proof of Proposition \ref{prop-sharp-l2-mean-curv} for $n=1$, \eqref{ineq-sharp-sufficient} and $\|X\|_2\cdot \|H\|_2\ge 1$ hold. By repeating the proof of Theorem \ref{thm-sharp-lambda1}, we conclude \eqref{ineq-Heintze} for curves with length $\le \frac{2\pi}{\sqrt{-\d}}$. 
\end{proof}

Since it may be helpful for further study on the totally squared curvature in a negatively curved space, we give more situations that yield \eqref{ineq-total-squared-curvature}

\begin{pro}\label{prop-curve-case-*}
	Let $N$ be a Cartan-Hadamard manifold with sectional curvature bounded above by $\d<0$. A regular closed curve $\gamma$ in $N$ satisfies \eqref{ineq-total-squared-curvature}
	if one of the following conditions holds.
	
	\begin{enumerate}\numberwithin{enumi}{theo}
		\item\label{prop-curve-1} $\gamma$ lies in a metric sphere $S(p_0,R)$, where $p_0$ is the center of mass of $\gamma$.
		\item\label{prop-curve-2} $\gamma$ does not lie in any metric sphere $S(p_0,R)$ but $$\|X\|_2^2\(1+\frac{\|X^\bot\|_2}{\|X\|_2}\)\le \frac{1}{-\d},$$
		where $X$ is the position vector about the center of mass $p_0$.
	\end{enumerate}
		
	Furthermore, there are curves of length $>\frac{2\pi}{\sqrt{-\d}}$ in a simply connected manifold of constant curvature $\d<0$ that satisfy \eqref{ineq-total-squared-curvature} but violate both \eqref{prop-curve-1} and \eqref{prop-curve-2}.
\end{pro}

Note that \eqref{prop-curve-2} makes sense for curves with a large length, but twisted with respect to the radial direction from $p_0$ so much that $\|X^\bot\|_2$ is very small.

\begin{proof}

For \eqref{prop-curve-1}, let $\gamma$ be in $S(p_0,R)$. Then $X^\bot=X$ and \eqref{ineq-l2-mean-curv-key} directly implies that $\|X\|_2\cdot \|H\|_2\ge 1$. By repeating the proof of Theorem \ref{thm-sharp-lambda1}, we conclude \eqref{ineq-total-squared-curvature}.

If equality in \eqref{ineq-total-squared-curvature} holds for curves satisfying \eqref{prop-curve-1} or \eqref{prop-curve-2}, then by the proof of the rigidity of Proposition \ref{ineq-total-squared-curvature}, $\gamma$ is a geodesic in a metric sphere.

For \eqref{prop-curve-2}, let $\gamma$ be a regular closed curve not lying in any metric sphere centered at $p_0$. By \eqref{ineq-l2-mean-curv-key},
\begin{equation}\label{ineq-general-curve}
\|X\|_2\|X^\bot\|_2\|H\|_2\ge \|X\|_2\(1 + \d \|X^\top\|_2^2\).
\end{equation}
For simplicity, let us denote $h=\|H\|_2$, and $\|X^\top\|_2=b\|X\|_2=bx$, then $0\le b\le 1$ is a constant. Then \eqref{ineq-general-curve} is written as 
$$\sqrt{1-b^2}xh\ge 1+\d b^2x^2.$$
In order to guarantee $xh=\|X\|_2\|H\|_2\ge 1$, it suffices
$$1+\d b^2x^2\ge \sqrt{1-b^2}.$$
Since $b\neq 0$, it is equivalent to
$$x^2(1+\sqrt{1-b^2})\le \frac{1}{-\d},$$
which is a reformation of the condition in \eqref{prop-curve-2}. 

We now prove that in general there are curves satisfying \eqref{ineq-l2-mean-curv-key} but violating \eqref{prop-curve-1} and \eqref{prop-curve-2}.
Let $\gamma$ be a regular curve in $\mathbb H^n$ with a very large length but uniformly bounded $\int \kappa^2 ds$ that violates \eqref{ineq-l2-mean-curv-key}, as \cite[Figure 8]{Langer-Singer1984} shows. Let $p_0$ be the center of mass of $\gamma$ and $X$ the position vector defined by \eqref{ineq-def-position-vec}. By \eqref{prop-curve-1}, $\gamma$ does not lie in any sphere at $p_0$. Let us shrink $\gamma$ along the radial direction from $p_0$ to a family of curves $\gamma_{\mu}$, defined by $X\circ\gamma_\mu(t)=\mu X(\gamma(t))$. Then the center of mass of $\gamma_\mu$ still is $p_0$, and $\gamma_{\mu}$ never lies in a metric sphere at $p_0$. By Theorem \ref{thm-total-squared-curvature}, there is a maximal $\mu_0>0$ such that \eqref{ineq-total-squared-curvature} holds for $\gamma_{\mu_0}$.  We will show that the condition on averaged position in \eqref{prop-curve-2} fails for $\gamma_{\mu_0}$.

Since \eqref{ineq-1st-eigen-mean-curv} for $n=1$ holds, we have $\|X\|_2\|H\|_2=1$ at $\gamma_{\mu_0}$. In order to see that $\|X\|_2^2\(1+\frac{\|X^\bot\|_2}{\|X\|_2}\)> 1$,
by the proof of \eqref{prop-curve-2}, it suffices to show that the strict inequality holds in \eqref{ineq-l2-mean-curv-key} for $\d=-1$, i.e.,
$$\|X^\bot\|_2\|H\|_2> 1- \|X^\top\|_2^2.$$

Let us argue by contradiction. The equality of \eqref{ineq-l2-mean-curv-key} implies that the induced metric on the parametrized surface consists of minimal geodesics from $p_0$ to $\gamma$ admits a form $d\rho^2+f^2d\theta^2$, where $f=c\sinh(\rho)$ depends only on $\rho$.
Then by direct calculation,
\begin{align*}
H(s)&=\nabla_{\gamma'(s)}\gamma'(s)=\(\rho''-(\theta')^2ff_\rho\)\frac{\partial}{\partial \rho}+\(\theta''+2\rho'\theta'\frac{f_\rho}{f}\)\frac{\partial}{\partial \theta}.
\end{align*}
At the same time, by the equality in Cauchy-Schwartz inequality there is a constant such that $H=aX^\bot$, where
\begin{align*}
X^\bot & =\frac{\sinh\rho}{\cosh \rho}\(1-(\rho')^2\)\frac{\partial}{\partial \rho}-\frac{\sinh\rho}{\cosh \rho}\rho'\theta' \frac{\partial }{\partial \theta}.
\end{align*}
Together with $\theta'^2=1-\rho'^2$, it follows that
\begin{align}
\label{ineq-special-1}
\rho''=\(1-(\rho')^2\)\(ff_\rho+a\frac{\sinh\rho}{\cosh\rho}\)\\
\label{ineq-special-2}
\theta''=-\rho'\theta'\(2\frac{f_\rho}{f}+a \frac{\sinh\rho}{\cosh \rho}\)-(\theta')^2\frac{f_\theta}{f}.
\end{align}
Since $\theta'\theta''=-\rho'\rho''$, multiplying $\theta'$ on both sides of \eqref{ineq-special-2} yields
\begin{equation}\label{ineq-special-3}
\rho''=\(1-(\rho')^2\)\(2\frac{f_\rho}{f}+a \frac{\sinh\rho}{\cosh \rho}\).
\end{equation}
Note that \eqref{ineq-special-1} together with \eqref{ineq-special-3} implies $f_\rho\equiv 0$. This contradicts to our choice of $\gamma_{\mu_0}$.
\end{proof}

The proof of Proposition \ref{prop-curve-case-*} shows that, the rigidity as in \eqref{ineq-sharp-l2-mean-curv} must fail for a closed curve with a large length that are not covered by \eqref{prop-curve-1} and \eqref{prop-curve-2}. The losing of rigidity also explains why the comparison method in this paper is not enough to solve Conjecture \ref{conj-main}. 

Let $\gamma$ be a closed regular curve in a Cartan-Hadamard manifold with sectional curvature $\le \d<0$. Instead of a uniform lower bound as in Conjecture \ref{conj-main}, lower bounds depending on the length $l$ of $\gamma$ can be derived from known results. Since \eqref{ineq-Heintze} always holds for $\d=0$, equivalently we have $\int_{\gamma}\kappa^2 ds\ge \frac{4\pi^2}{l}$.
It was conjectured by N. H. Kuiper in 1973, and confirmed by \cite{Brickell-Hsiung1974} and \cite{Tsukamoto1974} independently that the total curvature $\int_{\gamma}|\kappa|>2\pi$. Hence by the Cauchy-Schwartz inequality $\int_\gamma\kappa^2>\frac{4\pi^2}{l}$. 

In the following we give a better lower bound of the total squared curvature, which takes an optimal minimum at $l=2\pi \sqrt{-\d}$.

\begin{pro}\label{pro-bound-large-length}
	Let $\gamma$ be a regular closed curve in a Cartan-Hadamard manifold with sectional curvature bounded above by $\d<0$. Then
	\begin{equation}\label{ineq-curve-general}
	\int_{\gamma}\kappa^2ds\ge \min\left\{\frac{4\pi^2}{l}-\d l, \frac{8\pi^2}{l}\right\},
	\end{equation}
	where $l$ is the length of $\gamma$.
\end{pro}

\begin{proof}
	Let $p_0$ be the center of mass of $\gamma$, and $X$ be the position vector from $p_0$. By the discussion before Theorem \ref{thm-total-squared-curvature}'s proof, it suffices to consider the case $\|X\|_2^2>\frac{1}{-2\d}$. Then by \eqref{ineq-1st-eigen-mean-curv} and \eqref{ineq-l2-mean-curv-key},
	$$\frac{4\pi^2}{l^2} \|X\|_2^2\le \d \|X^\bot\|_2^2+\|X^\bot\|_2\cdot \|H\|_2.$$
	This yields
	$$\frac{\|H\|_2^2}{\|X\|_2^2}\ge \(\frac{4\pi^2}{l^2} \cdot\frac{\|X\|_2}{\|X^\bot\|_2}+(-\d) \frac{\|X^\bot\|_2}{\|X\|_2}\)^2\ge \frac{-16 \d \pi^2}{l^2}.$$
	Since $\|X\|_2^2> \frac{1}{-2\d}$, we derive $\|H\|_2^2=\frac{1}{l}\int_{\gamma}\kappa^2ds > \frac{8\pi^2}{l^2}$.
\end{proof}

Let us give some further remarks on Conjecture \ref{conj-main} and the inequality \eqref{ineq-total-squared-curvature}.

Because $\int_{\gamma}\kappa^2ds$ is the bending energy of an inextensible wire $\gamma$, a critical point of this functional was called a (free) elasticae by Euler in 1744, who solved the problem for curves in the plane (see \cite{FKN2016} for a detailed introduction, and \cite{Truesdell1983} for a brief historical discussion). Since then, elastica problems have attracted generations of mathematicians and physicists. 

A complete classification of elastica (with fixed arclength) in the sphere and hyperbolic plane was given by Langer and Singer \cite{Langer-Singer1984} in the same paper mentioned before, and independently by Bryant and Griffiths \cite{Bryant-Griffiths1986}. The optimal lower bound \eqref{ineq-Langer-Singer} essentially comes from the behavior of all elastica in the hyperbolic plane as investigated in \cite{Langer-Singer1984}, and it had been used to prove the Willmore conjecture for tori of revolution \cite{Langer-Singer1984-2}.

The curve-straightening flow of $\gamma$ along the negative gradient of bending energy with fixed length was also extensively studied (\cite{Langer-Singer1985,Langer-Singer1987},\cite{Linner1998}, etc). However, phenomena happens in the hyperbolic plane are complicate. It is unknown whether the curve-straightening takes any closed curve to the equator of $\mathbb H^2$, i.e., the global minimum of $\int_{\gamma}\kappa^2ds$, since the Palais-Smale condition fails to hold; see \cite{Steinberg1999}, \cite{Linner1998} and \cite{Singer2008} for more discussions. A solution to Conjecture \ref{conj-main} via curve-straightening seems to be far away from being reachable in a Cartan-Hadamard manifold.

By analyzing the behavior of elastica, it was proved in \cite{Langer-Singer1984} that there is $L>2\pi$ such that \eqref{ineq-total-squared-curvature} holds for any regular \emph{simple} closed curve in $\mathbb H^2$ with length $< L$. We do not know whether it is true in a Cartan-Hadamard manifold of negative curvature.

\begin{prob}
	Given any $\mu\le -1$, is there $L_0(\mu)>1$ such that \eqref{ineq-total-squared-curvature} holds for any regular null-homotopic simple closed curves with length $\le L_0(\mu)$ in a Riemannian manifold $N$ with sectional curvature $\mu \le K_N\le -1$?
\end{prob}

\end{document}